\documentclass[a4paper, 11pt]{amsart}

\usepackage[utf8]{inputenc}
\usepackage{amsmath, amssymb, amsfonts}
\usepackage{mathrsfs}

\usepackage[dvips,%
    includehead,%
    includefoot,%
    nomarginpar,%
    lmargin=1.2in,%
    rmargin=1.2in,%
    tmargin=1.4in,%
    bmargin=1.4in,%
  ]{geometry}

\usepackage{xcolor}
\usepackage[colorlinks,
    linkcolor={blue!60!black},
    citecolor={blue!60!black},
    urlcolor={blue!60!black}]{hyperref}

\usepackage{enumerate,enumitem}
\usepackage{tikz}
\usetikzlibrary{calc, matrix, arrows, cd, decorations.markings, knots}
\usetikzlibrary{decorations.pathmorphing}
\usetikzlibrary{decorations.pathreplacing}

\let\fullref\autoref

\theoremstyle{plain}
\newtheorem{proposition}{Proposition}[section]
\newtheorem{theorem}{Theorem}[section]
\newtheorem{corollary}{Corollary}[section]
\newtheorem{lemma}{Lemma}[section]

\theoremstyle{definition}
\newtheorem{definition}{Definition}[section]
\newtheorem{notation}{Notation}[section]
\newtheorem{question}{Question}[section]

\newtheorem{example}{Example}[section]

\theoremstyle{remark}
\newtheorem{remark}{Remark}[section]

\makeatletter
\let\c@corollary=\c@theorem
\let\c@proposition=\c@theorem
\let\c@lemma=\c@theorem
\let\c@remark=\c@theorem
\let\c@definition=\c@theorem
\let\c@notation=\c@theorem
\let\c@construction=\c@theorem
\let\c@example=\c@theorem
\let\c@equation\c@theorem
\let\c@qyestion\c@theorem
\makeatother

\def\makeautorefname#1#2{\expandafter\def\csname#1autorefname\endcsname{#2}}
\makeautorefname{theorem}{Theorem}%
\makeautorefname{corollary}{Corollary}%
\makeautorefname{proposition}{Proposition}%
\makeautorefname{lemma}{Lemma}%
\makeautorefname{remark}{Remark}%
\makeautorefname{definition}{Definition}%
\makeautorefname{notation}{Notation}%
\makeautorefname{section}{Section}%
\makeautorefname{construction}{Construction}%
\makeautorefname{example}{Example}%
\makeautorefname{question}{Question}%

\newcommand{\ol}{\overline}
\newcommand{\sm}{\setminus}

\newcommand{\Q}{\mathbb{Q}}
\newcommand{\Z}{\mathbb{Z}}
\newcommand{\N}{\mathbb{N}}
\newcommand{\C}{\mathbb{C}}
\newcommand{\R}{\mathbb{R}}
\newcommand{\F}{\mathbb{F}}

\newcommand{\im}{\operatorname{Im}}
\newcommand{\ord}{\operatorname{ord}}
\newcommand{\tr}{\operatorname{tr}}
\newcommand{\Id}{\operatorname{Id}}
\newcommand{\GL}{\operatorname{GL}}

\newcommand{\Ext}{\operatorname{Ext}}

\newcommand{\Herm}{\operatorname{Herm}}
\newcommand{\Hom}{\operatorname{Hom}}

\newcommand{\pt}{\operatorname{pt}}
\newcommand{\sgn}{\operatorname{sgn}}

\newcommand{\cone}{\operatorname{cone}}
\newcommand{\coker}{\operatorname{coker}}

\newcommand{\eps}{\varepsilon}
\renewcommand{\epsilon}{\varepsilon}
\renewcommand{\phi}{\varphi}

\begin{document}
\title{Doubly slice knots and metabelian obstructions}

\author[Patrick Orson]{Patrick Orson}
\address{Department of Mathematics, Boston College, USA}
\email{patrick.orson@bc.edu}

\author[Mark Powell]{Mark Powell}
\address{Department of Mathematical Sciences, Durham University, UK}
\email{mark.a.powell@durham.ac.uk}

\begin{abstract}
For $\ell >1$, we develop $L^{(2)}$-signature obstructions for $(4\ell-3)$-dimensional knots with metabelian knot groups to be doubly slice. For each $\ell>1$, we construct an infinite family of knots on which our obstructions are non-zero, but for which double sliceness is not
obstructed by any previously known invariant.
\end{abstract}

\maketitle

\section{Introduction}

All manifolds considered in this article are topological and submanifolds are the images of locally flat embeddings (unless otherwise stated). An \emph{$n$-knot} is an oriented submanifold of $S^{n+2}$ homeomorphic to $S^n$.
An $n$-knot is \emph{slice} if it is ambiently isotopic to the equatorial cross section of some $(n+1)$-knot $J$, and is moreover  \emph{doubly slice} if $J$ may be taken to be the $(n+1)$-dimensional unknot.

For $n>1$, the question of which knots are slice is considered solved. All even-dimensional knots are slice~\cite{Kervaire:1965-1}, and odd-dimensional knots with $n>1$ are slice if and only if they are \emph{algebraically slice} (\fullref{def:algslice})~\cite{MR0246314}.  Kervaire and Levine worked in the smooth category, but after Marin provided topological transversality and Kirby-Siebenmann proved that codimension two locally flat embeddings have normal bundles, the Kervaire-Levine result also holds in the topological category.

The doubly slice problem is unsolved in every dimension. This is particularly striking for $n>1$, considering the effectiveness of classical Surgery Theory for studying codimension 2 embedding problems. The first stage to the Kervaire-Levine solution to the slice problem is to show that every $n$-knot $K$ with $n >1$ is concordant to a knot with knot group $\Z$. Because this can always be done, whether such a knot is slice has nothing to do with the particular knot group and only depends on the abelianisation. While similar abelian invariants do obstruct double sliceness, no such general simplification is possible for the doubly slice problem. This fact was first exploited by Ruberman~\cite{Ruberman-doubly-slice-I, Ruberman-doubly-slice-II} to provide examples in every dimension of algebraically doubly slice knots that are not doubly slice. These use the existence of non abelian representations of some knot groups, in the manner of the Casson-Gordon invariants~\cite{Casson-Gordon:1986-1}. The effectiveness of these obstructions highlights an interesting similarity between the high-dimensional doubly slice problem and the low-dimensional slice problem. In this paper we continue the study of non abelian doubly slice obstructions. Our main theorem is the following.

\begin{theorem}\label{thm:family}For all $\ell>1$, there exists an infinite family of mutually non ambiently isotopic, non doubly slice $(4\ell-3)$-knots~$\{K_i\}$, such that for each $K_i$:
\begin{enumerate}[leftmargin=*]
\item\label{item:main-thm-1} There is a $\Z[\Z]$-homology equivalence between the exterior of $K_i$ and the exterior of a doubly slice knot $J$, that is the identity on the boundary, so particular preserves meridians.
\item\label{item:main-thm-2} Ruberman's Casson-Gordon invariants do not obstruct $K_i$ from being doubly slice.
\end{enumerate}
\end{theorem}

Point (\ref{item:main-thm-1}) implies that the knots $K_i$ have hyperbolic Seifert and Blanchfield pairings, so are algebraically doubly slice in the sense of Sumners \cite{MR0290351}. In particular every $K_i$ is slice. Points (\ref{item:main-thm-1}) and (\ref{item:main-thm-2}) combine to show that no previously known obstruction could detect the fact that these knots are not doubly slice.

\subsection*{Metabelian groups, $L^{(2)}$ signature obstructions and concordance}
Recall that for a group $G$, the \emph{$i$th derived subgroup} $G^{(i)}$ is defined inductively by~$G^{(0)}:=G$ and $G^{(i+1)}:=[G^{(i)},G^{(i)}]$. We say a group $G$ is \emph{metabelian} if $G^{(2)}=0$. In other words, a group is metabelian if it has abelian commutator subgroup.

The classical \emph{algebraic concordance} invariants used by Levine~\cite{MR0246314} rely only on the representation of the knot group given by abelianisation. The abelianisation of a group is the first subquotient of the derived series of the group. For $1$-knots, it is now known that \emph{each further} subquotient of the derived series of the knot group can be used to provide obstructions beyond those of algebraic concordance~\cite{MR1973052}.  In this way the topological knot concordance group can be filtered. The first examples of non abelian concordance obstructions were the Casson-Gordon invariants~\cite{Casson-Gordon:1986-1}, which may be interpreted as metabelian-level obstructions in the Cochran-Orr-Teichner filtration of the (topological) knot concordance group~\cite[\textsection 9]{MR1973052}.

To obstruct deeper stages of the filtration, Cochran, Orr and Teichner developed the use of $L^{(2)}$ signature defects, also called $L^{(2)}$ $\rho$-invariants. The use of $L^{(2)}$ cohomology for the study of non compact manifolds has a long history, for which we refer the reader to the book of L\"{u}ck \cite{Lueck:2002-1}. The use of these techniques to study knot concordance began with the work of Cochran-Orr-Teichner, where the non compact manifold in question is the covering space of the knot exterior corresponding to a representation of the the knot group to an infinite group. These techniques were further developed by Cha-Orr and by Cha~\cite{Cha-Orr:2009-01,MR3257550}. Since Cochran-Orr-Teichner, these ideas have been successfully applied to the study of concordance of $1$-knots by many authors, including \cite{Kim:2005-1}, \cite{Cochran-Kim:2004-1}, \cite{Cochran-Harvey-Leidy:2009-1}, \cite{Cochran-Harvey-Leidy:2011-02}, and~\cite{Franklin:2013}. Particularly relevant to this paper is the work of Taehee Kim~\cite{Kim:2006-1}, who used $L^{(2)}$ $\rho$-invariants to obstruct double sliceness of $1$-knots.

Inspired by these techniques, and by the work of Ruberman, this paper investigates further applications of low-dimensional slice obstructions to the high-dimensional doubly slice problem. We will focus on $n=4\ell-3$, where $\ell>1$. When the fundamental group of the knot complement is metabelian, we develop new obstructions to being doubly slice, based on $L^{(2)}$ $\rho$-invariants, and analogous to Kim's obstructions~\cite{Kim:2006-1}. We use Wall realisation to produce an infinite family of non isotopic knots in dimensions $4\ell-3>1$. The knots in this family are obstructed from being doubly slice by our invariants. By making careful choices in our construction, we prove that our knots cannot be shown to be non doubly slice by any previous method.

\subsection*{Further questions}
The most classical obstruction to double sliceness is \emph{algebraic double sliceness} (\fullref{def:algslice}), which is an abelian obstruction. In Section~\ref{section:first-order-obstructions} we will discuss abelian obstructions to double sliceness more generally, proving that (\ref{item:main-thm-1}) implies the knots~$K_i$ from \fullref{thm:family} are algebraically doubly slice in the stronger algebraic sense of~\cite{Levine-doubly-slice} and in the stronger algebraic sense of~\cite{MR3604378}. The subject of this paper is the existence of doubly slice obstruction beyond the abelian, but (stepping backwards a little) the class of knots for which even these strong algebraic doubly slice conditions characterise double sliceness is still unknown. More precisely the following is open.

\begin{question}\label{q:1} Is an algebraically doubly slice knot with knot group $\Z$ necessarily doubly slice? Is an algebraically doubly slice knot strongly algebraically doubly slice?
\end{question}

There are two natural generalisations of the results of this paper. They both present significant technical challenges.

\begin{question}\label{q:2} When $n\not\equiv 1\pmod 4$, do there exists strongly algebraically doubly slice $n$-knots with metabelian knot groups that are not doubly slice but for which Ruberman's invariants are ineffective for detection?
\end{question}

\begin{question}\label{q:3} Can further subquotients of the derived series of the knot group be exploited to obstruct double sliceness high-dimensionally?
\end{question}

The challenge in answering \fullref{q:2} is that in these dimensions modulo 4, there is no developed obstruction theory coming from $L^{(2)}$ cohomology. 

Considering the filtration framework established in \cite{MR1973052}, answering \fullref{q:3} seems like another natural next step. But that the knot group was metabelian was important to the proof of our main obstruction \fullref{thm:main} because of the way the knot group interacts with the slice disc groups. The challenge of extending our techniques beyond the metabelian level is that the fundamental groups of the exteriors of complementary slice discs for a doubly slice knot do not necessarily interact with representations of the knot group in such a natural way.

\subsection*{Acknowledgements}
We thank Stefan Friedl and Matthias Nagel for helpful discussions.


\section{Conventions and preliminary results}

For $X$ a manifold and $Y\subset X$ a proper, neat submanifold of codimension either 1 or 2 there is an open tubular neighbourhood; see \cite{MR0645390} for ambient dimension $\neq 4$ and~\cite[Section 9.3]{Freedman-Quinn:1990-1} for $\dim X =4$. Denote such a neighbourhood by $\nu Y\subset X$.

For $G$ a
group, the group ring $\Z[G]$ has an involution $\lambda\mapsto \ol{\lambda}$ defined by linearly extending the involution $\ol{g}:=g^{-1}$. For a left $\Z[G]$-module $A$, let $\ol{A}$ denote the right ${\Z[G]}$-module with the same underlying abelian group and with the ${\Z[G]}$-action $a\cdot\lambda=\ol{\lambda}\cdot a$. Modules are \emph{left} modules unless specified otherwise. Write $\Herm_{r\times r}(\Z[G])$ for the group of $r\times r$ matrices $U$ with values in $\Z[G]$ and such that $U^T=\overline{U}$. We say $U\in\Herm_{r\times r}(\Z[G])$ is \emph{non-degenerate} if the corresponding endomorphism of $(\Z[G])^r$ is injective, and \emph{nonsingular} if it is moreover an isomorphism.

Let $K$ be an $n$-knot. Write $X_K := S^{n+2} \sm \nu K(S^n)$ for the \emph{exterior}. Write $\pi_K := \pi_1(X_K)$ for the \emph{group of the knot}. Write $M_K:=X_K\cup_{S^1\times S^n} S^1\times D^{n+1}$ for the  \emph{surgery manifold}. When $n=1$, this glueing is specified by using the $0$-framing of the knot. When $n>1$, any homeomorphism of $S^1\times S^n$ extends over $S^1\times D^{n+1}$, so $M_K$ is well-defined up to homeomorphism. Indeed,
Gluck, Browder and Kato~\cite{Gluck, MR212816, Kato-concordance} showed that for every $n \geq 2$, the group of pseudo-isotopy classes of PL-homeomorphisms of $S^1 \times S^n$ is $(\Z/2)^3$, generated by reflections in the $S^n$ and $S^1$ factors, and the Gluck twist $(\theta,x) \mapsto (\theta,\rho(\theta)\cdot x)$, with $\rho \colon S^1 \to SO(n+1)$ a homotopically essential map. Each of these maps extends over $S^1 \times D^{n+1}$, and so every PL-homeomorphism of $S^1 \times S^n$ extends.  A homeomorphism $f \colon M \to N$ between closed PL manifolds is homotopic through homeomorphisms to a PL homeomorphism if and only if the Casson-Sullivan invariant  $\kappa(f) \in H^3(M;\Z/2)$ vanishes~\cite{hauptvermutung-book}.  For $n=4\ell-3$, $H^3(S^1 \times S^n;\Z/2)=0$, so every homeomorphism is homotopic to a PL homeomorphism. It follows that every homeomorphism of $S^1 \times S^n$ extends to a homeomorphism of $S^1 \times D^{n+1}$.

 By Alexander duality, $X_K$ has the homology of $S^1$ and hence the abelianisation $\pi_1(X_K)\to \Z$ determines an infinite cyclic cover $\ol{X}_K\to X_K$. The inclusion $X_K\subset M_K$ induces an isomorphism $\pi_1(X_K)\cong \pi_1(M_K)$ and so there is also an infinite cyclic cover $\ol{M}_K\to M_K$. With respect to these covers, we take homology with coefficients in $\Z[\Z]$ and, as $\ol{M}_K$ is of the form $\ol{M}_K\cong \ol{X}_K\cup_{\R\times S^n}\R\times D^{n+1}$, we have that $H_*(X_K;{\Z[\Z]})\cong H_*(M_K;{\Z[\Z]})$.

For a ${\Z[\Z]}$-module $A$, let $T_{\Z} A\subset A$ denote the $\Z$-torsion submodule and let $FA:=A/T_{\Z} A$ denote the $\Z$-torsion free quotient module.

\begin{definition}
Write  $S := \{p(t)\in\Z[t,t^{-1}] \mid p(1)=\pm1 \} \subset \Z[\Z]$ for the \emph{Alexander polynomials}.
\end{definition}

A finitely generated $\Z[\Z]$-module $A$ is \emph{torsion with respect to $S$} if there exists $p\in S$ such that $pA=0$. The module $A$ is torsion with respect to $S$ if and only multiplication by $1-t$ is an automorphism of $A$; see e.g.~\cite[Lemma 4.5]{MR3604378}. The reduced homology $\widetilde{H}_*(X_K;{\Z[\Z]})$ is torsion with respect to $S$; see~\cite[1.3]{Levine-77-knot-modules}. Writing $\Q(\Z)$ for the field of fractions of~${\Z[\Z]}$, this implies that the reduced homology $\widetilde{H}_*(M_K;\Q(\Z))\cong \widetilde{H}_*(X_K;\Q(\Z))=0$. We will call a finitely generated $\Z[\Z]$-module chain complex \emph{$S$-acyclic} if the homology is torsion with respect to $S$.

A \emph{slice disc} for an $n$-knot $K$ is a proper, neat submanifold $(D,K)\subset (D^{n+3},S^{n+2})$, homeomorphic to $(D^{n+1},S^{n})$.
Recall that $K$ is slice if and only if $K$ is the boundary of a slice disc, and that $K$ is doubly slice if and only if $K$ is the boundary of two slice discs $D_\pm$ that are \emph{complementary}, in the sense that $D_+\cup_K D_-$ is unknotted in $S^{n+3}$.

\subsection{Cyclic branched covers of $n$-knots}
Let $r\in\N$ and let $K$ be an $n$-knot. Denote by $\Sigma_r(K)$ the \emph{$r$-fold cyclic branched cover} of $S^{n+2}$ branched over $K$. Later in this paper we would like to use several results about cyclic branched covers from the low-dimensional topology literature. We now state these results and check they still apply high-dimensionally.

\begin{notation}\label{not:primed}Given a group homomorphism $\xi\colon \pi\to G$ to a finite group $G$, define $\xi'\colon \pi\to G\to U(\C [G])$, the canonically induced unitary representation given by left multiplication. Note $\C [G]$ is a finite dimensional vector space of dimension $|G|$.
\end{notation}

\noindent We record a straightforward lemma for later use.

\begin{lemma}\label{lem:obvs} Let $\pi$ be group and $T$ be a finitely generated $\Z[\pi]$-module. Suppose $\phi\colon \pi\to G$ is a homomorphism to a finite group. Then there is a natural decomposition of complex vector spaces
\[
\C[G]\otimes_{\phi'}T\cong \sum_{\alpha}V_\alpha^{\dim (\alpha)}\otimes_{\alpha\circ\phi'} T
\]
where $\alpha$ ranges over all conjugacy classes of complex irreducible unitary representations of $G$, and $V_\alpha$ is the irreducible representation at $\alpha$.
\end{lemma}

\begin{proof}
The left regular representation for a finite group $G$ is conjugate to the sum of representations
\[
\C[G]\cong\sum_{\alpha}V_\alpha^{\dim (\alpha)}.
\]
\end{proof}

For the convenience of the reader, we provide a proof of the following well known fact, originally observed by Fox \cite{Fox:1956-1}.  We were not able to find a complete proof in the literature (Gordon~\cite{Gordon-branched, MR521730} comments that Fox's proof requires modification), so we provide one here.

\begin{proposition}\label{prop:fox}
Suppose $0\to P\xrightarrow{A} P\to T\to 0$ is a short exact sequence of $\Z[t,t^{-1}]$-modules such that $P$ is finitely generated and free. Let $C_r=\{e,x,\dots, x^{r-1}\}$ denote the cyclic group of order $r\in\N$, and define a homomorphism $\phi\colon \Z\langle t\rangle\to C_r$ by $\phi(t)=x$.  Then $\Z[C_r]\otimes_\phi T$, considered as an abelian group, has order
\[
\pm\prod_{j=0}^{r-1}\det(A(\exp(2\pi ij/r))),
\]
where if this product is 0, we interpret this as infinite order.
\end{proposition}

\begin{proof}
Observe that, writing $k$ for the free rank of $P$ and choosing a free basis,  $A$ can be written as a sum $A(t)=\sum_{-\infty}^\infty A_i t^i$ for some $A_i\in\GL(k,\Z)$, such that $A_i$ is nonzero for only finitely many $i$. We will use this later.

The presentation of $T$ determines an exact sequence of $\Z[C_r]$-modules
\[
 \Z[C_r]\otimes_\phi P \xrightarrow{\Id\otimes_\phi A} \Z[C_r]\otimes_\phi P\to \Z[C_r]\otimes_\phi T\to 0.
\]
Choose an isomorphism $\Z[C_r]\cong \Z^r$ of free $\Z$-modules. We may now view $\Id\otimes_\phi A$ as an endomorphism of a rank $rk$ free abelian group. The determinant of this endomorphism is nonzero if and only if $\Z[C_r]\otimes_\phi T$ is a finite abelian group, in which case the absolute value of the determinant computes the order of the group.

We calculate this determinant by passing to complex coefficients and considering $\C[G]\otimes_\phi P\cong \C^r\otimes_\phi P$. Apply \fullref{lem:obvs} with $\pi=\Z$, $G=C_r$, noting that the irreducible unitary representations of $C_r$ are the complex 1-dimensional representations $\chi_0, \chi_1,\dots,\chi_{r-1}$ where $\chi_j(x)=\omega_r^j$. We write $\phi_j=\chi_j\circ\phi'$. Thus as an endomorphism of the vector space $\C^r\otimes_\phi P$, we calculate that
\begin{align*}
\det(\Id\otimes_\phi A)&=\det\Big(\Id\otimes_\phi\Big(\sum_i A_it^i\Big)\Big)=\det\Big(\sum_i x^i\otimes_\phi A_i\Big) \\
&=\det\bigg(\bigoplus_{j=1}^{r-1}\Big(\sum_i\omega_r^{ij}\otimes_{\phi_j} A_i\Big)\bigg)
=\prod_{j=0}^{r-1}\det\Big(\sum_i(\omega_r^j)^i\otimes_{\phi_j} A_i\Big) \\
&=\prod_{j=0}^{r-1}\det\left(\Id_\C\otimes_{\phi_j} A(\omega_r^j)\right)=\prod_{j=0}^{r-1}\det\left(A(\omega_r^j)\right),
\end{align*}
as claimed.
\end{proof}

\begin{definition}
For a knot $K$ define the \emph{Alexander polynomial} to be any generator of the minimal principal ideal that contains the first elementary ideal of $H_1(\ol{X}_K;\Z)$. The polynomial $\Delta_K(t)$ is well-defined up to units in $\Z[t,t^{-1}]$.
\end{definition}

\begin{remark}\label{rem:determinant}
In the case that the first elementary ideal of $H_1(\ol{X}_K;\Z)$ is principal, as is always the case for $1$-knots, there is an exact sequence
\[
0\to P \xrightarrow{A} P\to H_1(\ol{X}_K;\Z)\to 0,
\]
where $P$ is a finitely generated free $\Z[t,t^{-1}]$-module. In this case, $\det(A)=\Delta_K(t)$.
\end{remark}

\begin{corollary}\label{prop:foxcheck} Let $K$ be an $n$-knot and suppose that $\pi_K$ has a deficiency 1 presentation. Then for all $r\in\N$,  the abelian group $H_1(\Sigma_r(K);\Z)$ has order
\[
\pm\prod_{j=0}^{r-1}\Delta_K(\exp(2\pi ij/r)),
\]
where if this product is 0, we interpret this as infinite order.
\end{corollary}

\begin{proof}
If a group $G$ has $G/G^{(1)}\cong \Z$, and $G$ has a deficiency 1 presentation then the first elementary ideal of $G^{(1)}/G^{(2)}$ is principal. The module $\Z[C_r]\otimes_\phi H_1(\ol{X}_K;\Z)$ is well-known to be equal to $H_1(\Sigma_r(K);\Z)$; see \cite[Section 5]{MR521730} for a proof when $n=1$ that easily generalises to $n\geq 1$. Combining this with \fullref{prop:fox} and \fullref{rem:determinant}, the result follows.
\end{proof}

The following theorem is a result of Livingston \cite[Theorem 1.2]{Livingston:2002-1} when $n=1$. The following high-dimensional version is also true.

\begin{theorem}\label{thm:livingston} Let $K$ be an $n$-knot such $\pi_K$ has a deficiency 1 presentation, and suppose that $r\in\N$ is a prime power. If all irreducible factors of $\Delta_K$ are cyclotomic polynomials $\Phi_m(t)$, with $m$ divisible by at least three distinct primes, then $H_1(\Sigma_r(K);\Z)=0$.
\end{theorem}

\begin{proof} Use \fullref{prop:foxcheck} and then proceed exactly as in the proof of \cite[Theorem 1.2]{Livingston:2002-1}.
\end{proof}

\subsection{Knots with metabelian group}
For $n\geq 3$, $n$-knots with a given knot group can be constructed using the following theorem of Kervaire \cite[Th\'{e}or\`{e}me I.1]{Kervaire:1965-1}.

\begin{theorem}[Kervaire]\label{thm:Kervaire}
All knot groups $\pi$ are finitely presented, with $H_1(\pi) \cong \Z$, $H_2(\pi) =0$, and $\pi$ normally generated by one element. When $n\geq 3$, for any group $\pi$ with these properties there exists a knot $K$ with group $\pi$.
\end{theorem}

Kervaire's construction even holds in the category of smooth knots. To work in maximal generality, we also state the next proposition in that category.

\begin{proposition}\label{prop:dannydouble} Let $n>3$ be an integer and let $\pi$ be a group satisfying the conditions in~\fullref{thm:Kervaire}. Then there exists a smooth doubly slice $n$-knot $K$ with $\pi_K=\pi$.
\end{proposition}

The additional fact that for $n>3$, $K$ may be taken to be doubly slice was observed by Ruberman \cite[Proof of Proposition~3.1]{Ruberman-doubly-slice-II}, but he did not give a proof, so we offer one here. Our proof will follow that of \cite[Th\'{e}or\`{e}me I.1]{Kervaire:1965-1}, with necessary modifications to obtain the desired result.

\begin{proof}
Let $\pi=\langle x_1,\dots,x_a \mid r_1,\dots,r_b\rangle$ be a group presentation of $\pi$, where $H_1(\pi) \cong \Z$, $H_2(\pi) =0$, and $\pi$ is normally generated by one element. Write $W_1$ for the boundary connected sum of $a$ copies of $S^1\times D^{n+1}$. Fix an isomorphism $\pi_1(W_1)\cong\langle x_1,\dots,x_a\rangle$ by sending the homotopy class of $S^1\times \pt$ in the $i$th copy of $S^1\times D^{n+1}$ to $x_i$. Note this also determines an isomorphism $\pi_1(\partial W_1)\cong\langle x_1,\dots,x_a\rangle$, and using this, choose $b$ disjoint embedded closed curves in $\partial W_1$ representing the words $r_j$ in the generators. As $\partial W_1$ is orientable, we may extend the embedded curves to disjointly embedded copies of $S^1\times D^{n}$. Use these tubular neighbourhoods as attaching regions and attach $2$-handles to $W_1$, to obtain $W_2$. Note that $\pi_1(W_2)\cong\pi$. By a theorem of Hopf, there is an exact sequence, $\pi_2(\partial W_2)\to H_2(\partial W_2;\Z)\to H_2(\pi)\to 0$. By hypothesis, $H_2(\pi)=0$, and hence any element of $H_2(\partial W_2;\Z)$ may be represented by a continuous map $S^2\to \partial W_2$. As the dimension of $\partial W_2$ is $n+2>4$, Whitney's embedding theorem means we may assume the elements of $H_2(\partial W_2;\Z)$ are represented by disjointly embedded $2$-spheres. Since the manifold $\partial W_2$ is stably parallelisable (being homotopy equivalent to a 2-complex), we may extend the 2-sphere embeddings to disjoint embeddings of copies of $S^2\times D^{n-1}$. Use these tubular neighbourhoods as attaching regions and attach $3$-handles to $W_2$, to obtain $W_3$. Note that $\pi_1(\partial W_3)=\pi_1(W_3)=\pi$ and that $H_r(\partial W_3)=0$ for $r>1$. We refer to the proof of~\cite[Th\'{e}or\`{e}me I.1]{Kervaire:1965-1} for the details.

Attach a $2$-handle to $W_3$ along an embedded $S^1\times D^{n}\hookrightarrow \partial W_3$ that represents a generator of $H_1(\partial W_3;\Z)\cong\Z$. This results in a contractible $(n+2)$-manifold $W$ with boundary a homotopy $(n+1)$-sphere. Write $\Delta\colon D^n\hookrightarrow W$ for the cocore of the final 2-handle attachment and $\Delta |_{\partial D^{n}}=J\colon S^{n-1}\hookrightarrow \partial W$. We note that $\pi_1(\partial W\sm J(S^{n-1}))=\pi_1(\partial W_3)=\pi$ and the inclusion $\partial W\sm J(S^{n-1})\subset W\sm \Delta(D^{n})$ induces an isomorphism on $\pi_1$. Define a locally flat embedding of $S^{n}$ by $-\Delta\cup_J\Delta\colon S^n\hookrightarrow -W\cup_{\partial W} W$.
Note that the closed $(n+2)$-manifold $-W\cup_{\partial W} W=:M$ is homotopy equivalent to $S^{n+2}$. For  $n>2$, the set of smooth closed homotopy $(n+2)$-spheres is a finite abelian group. Let $N$ be the group inverse of $M$, so that $N\#M$ is diffeomorphic to $S^{n+2}$, and $K:=\im(-\Delta\cup_J\Delta)\subset N\# M$ is a smooth $n$-knot.
By~\cite[Theorem B]{Levine-doubly-slice}, doubled disc knots are doubly slice, so $K$ is doubly slice. By the Seifert-Van-Kampen theorem, $\pi_K=\pi *_{\pi} \pi \cong \pi$.
\end{proof}

\noindent We record a well known fact about knot groups.

\begin{proposition}\label{prop:metabelian}
If a knot $K$ has metabelian fundamental group then $\pi_K\cong \Z\ltimes H_1(X_K;\Z[\Z])$, where the splitting depends on a choice of homomorphism $\Z \to \pi_K$, corresponding to a choice of oriented meridian.
\end{proposition}

\begin{proof}
For any group $G$, there is then an exact sequence
\[
0\to G^{(1)}/G^{(2)}\to G/G^{(2)}\to G/G^{(1)}\to 0.
\]
Now let $X$ be a topological space with $\pi_1(X)/\pi_1(X)^{(1)}\cong H_1(X;\Z)\cong \Z$. We see that the abelianisation of the commutator subgroup of $\pi_K$ is $\pi_1(X)^{(1)}/\pi_1(X)^{(2)}\cong H_1(X;{\Z[\Z]})$, with ${\Z[\Z]}$-coefficients determined by the abelianisation of $\pi_1(X)$. Setting $G=\pi_1(X)$, the sequence above splits. This gives
\[
\pi_1(X)/\pi_1(X)^{(2)}\cong \Z\ltimes H_1(X;{\Z[\Z]}).
\]
In particular, suppose $K$ is an $n$-knot such that $\pi_K \cong \pi_K/\pi_K^{(2)}$ is metabelian.  Let $H:= \pi_K^{(1)}/\pi_K^{(2)} = \pi_K^{(1)} = H_1(X_K;{\Z[\Z]})$.
Make a choice of oriented meridian $\mu\subset S^{n+2}$ for $K$. This determines a homomorphism $\Z = \pi_K/\pi_K^{(1)}\to \pi_K$ by $1\mapsto [\mu]$. This homomorphism splits the sequence
\[
0 \to H \to \pi_K \to \Z \to 0,
\]
on the right, resulting in an isomorphism $\pi_K \cong \Z \ltimes H$.
\end{proof}

The next result provides examples of metabelian knot groups for high-dimensional knots.

\begin{proposition}\label{lem:buildtheknot} Let $\mathcal{A}$ be a finitely generated ${\Z[\Z]}$-module such that multiplication by~$t-1$ acts as an automorphism on $\mathcal{A}$. For all $n\geq 3$, there is an $n$-knot $K$ such that the group of $K$ is the metabelian group $\Z \ltimes \mathcal{A}$, and for $n>3$ this knot may be chosen to be doubly slice.
\end{proposition}

\begin{proof}By \fullref{thm:Kervaire} and \fullref{prop:dannydouble}, it is sufficient to check that $\Z\ltimes\mathcal{A}$ satisfies Kervaire's conditions. The group $\Z \ltimes \mathcal{A}$ is normally generated by one element because $$(0,-h)(1,0)(0,h) = (1,(t-1)h),$$ and for any $h'\in \mathcal{A}$ there is an $h$ such that $(t-1)h=h'$.  Thus $(0,-h)(1,0)(0,h)(n-1,0) = (n,h')$, so $(n,h')$ can be realised as a product of conjugates of $(1,0)$.  Since~$t-1$ acts as an automorphism of $\mathcal{A}$, it certainly does on the homology of $\mathcal{A}$,~$t-1 = \Id \colon H_i(\mathcal{A};\Z) \to H_i(\mathcal{A};\Z)$, for $i > 0$.  For $i=0$, multiplication by $t-1$ induces the zero map.  Then the Wang sequence
  \begin{align*} & H_2(\mathcal{A}) \xrightarrow{t-1 =\Id} H_2(\mathcal{A}) \to H_2(\Z
  \ltimes \mathcal{A}) \\ \to &H_1(\mathcal{A}) \xrightarrow{t-1 =\Id} H_1(\mathcal{A}) \to H_1(\Z
  \ltimes \mathcal{A}) \\ \to & H_0(\mathcal{A}) \cong \Z \xrightarrow{t-1 =0} H_0(\mathcal{A}) \cong \Z \end{align*}
computes that the group homology $H_2(\Z\ltimes \mathcal{A}) = 0$ and $H_1(\Z\ltimes \mathcal{A}) \cong \Z$.
\end{proof}

\begin{example}
Let $p(t)\in\Z[t,t^{-1}]$ such that $p(1)=\pm 1$. Then $1-t$ acts on the module $\mathcal{A}=\Z[t,t^{-1}]/p(t)$ by automorphism. Moreover, $\mathcal{A}$ is finitely generated as an abelian group. This provides a good source of modules for \fullref{lem:buildtheknot}.
\end{example}

\begin{remark}
When $n=1$, the only knot with a metabelian group is the unknot. To see this, consider that when a $1$-knot has nontrivial Alexander polynomial, the longitude of the knot always lies in the group $\pi_K^{(2)}$, so by the loop theorem this group is non-vanishing and $\pi_K$ is not metabelian. Now if $K$ has trivial Alexander polynomial then $H_1(\ol{X}_K;\Z)=0$, so that $\pi_K/\pi_K^{(2)}\cong\Z\ltimes\{e\}$. If $\pi_K$ is also metabelian then $\pi_K\cong\Z$, meaning $K$ is the unknot. Our main doubly slice obstruction (\fullref{thm:main}) requires a metabelian fundamental group, so is not interesting for $1$-knots.
\end{remark}

\subsection{Strebel's class}

\noindent We collect some additional group theoretic definitions and result which we will use later.

\begin{definition}A group $G$ is \emph{poly-torsion-free-abelian (PTFA)} if there exists a subnormal series\[G=G_0\geq G_1\geq G_2\geq\dots\geq G_n=\{e\},\]such that each subquotient $G_i/G_{i+1}$ is torsion free abelian.
\end{definition}

\begin{definition}
A group $G$ is \emph{residually finite} if there exists a descending chain of normal subgroups $G\supset G_1\supset G_2\supset \dots$, each of finite index $|G/G_i| < \infty$, and such that $\bigcap_i G_i=\{e\}$
\end{definition}

\noindent
We recall a definition due to Strebel \cite[\textsection 1.1]{Strebel:1974-1}.

\begin{definition}Let $R$ be a ring with unit $1\neq 0$ and involution. Suppose for a group~$G$ that given any homomorphism $\theta\colon P\to Q$ of projective left~$RG$-modules such that the homomorphism $1_R\otimes_{RG}\theta\colon R\otimes_{RG} P\to R\otimes_{RG}Q$ is injective, we have that $\theta$ is injective. Then we say $G$ belongs to the \emph{Strebel class $D(R)$}.
\end{definition}

\noindent
The following is a consequence of results in \cite{Strebel:1974-1}.

\begin{lemma}\label{lem:strebelclass}
For all $R$, any PTFA group is in $D(R)$.
\end{lemma}
\begin{proof}By \cite[Proposition 1.5(ii)]{Strebel:1974-1}, it is enough to show that for all $R$ any torsion free abelian group $G\cong\Pi_i\Z$ is in $D(R)$. If for some $R$ and all $i\in I$, some index set, groups $G_i$ are in $D(R)$, then $\Pi_iG_i$ is in $D(R)$ \cite[Proposition 1.5(iii)]{Strebel:1974-1}. So it is enough to show that, for all $R$, $\Z$ is in $D(R)$. This is shown in \cite[Proposition 1.3]{Strebel:1974-1}.
\end{proof}

\section{Metabelian $L^{(2)}$ $\rho$-invariants for $(4\ell-1)$-manifolds}

Suppose $(M,\phi)$ consists of a closed, oriented, topological $(4\ell-1)$-manifold $M$ together with a map $\varphi\colon M\to BG$, where $G$ is a discrete group. We now recall how to obtain a topological invariant $\rho^{(2)}(M,\varphi)\in \R$ using an $L^{(2)}$-signature defect.

We recall some definitions and features of von Neumann algebras; see \cite[Chapter 1]{Lueck:2002-1} for a more complete account. Let $\Gamma$ be a countable discrete group and let $\ell^{2}\Gamma$ be the Hilbert space of square summable formal sums of group elements with complex coefficients. For a Hilbert space $H$, denote by $\mathcal{B}(H)$ the $C^*$-algebra of bounded linear operators from $H$ to itself. The \emph{group von Neumann algebra} of $\Gamma$ is a $\C$-algebra with involution $\mathcal{N}\Gamma$, consisting of the $\Gamma$-equivariant elements of $\mathcal{B}(\ell^{2}\Gamma)$. By convention, an element of $\C\Gamma$ determines a linear operator on $\ell^{2}\Gamma$ by left multiplication, and in this way there are inclusions of $\C$-algebras with involution $\C\Gamma\subseteq \mathcal{N}\Gamma\subseteq \mathcal{B}(\ell^{2}\Gamma)$.

Any projective $\mathcal{N}\Gamma$-module $P$ has an associated \emph{$\mathcal{N}\Gamma$-dimension} $\dim_{\mathcal{N}\Gamma}P\in[0,\infty)$; see~\cite[Chapter 6]{Lueck:2002-1}. A finitely generated projective $\mathcal{N}\Gamma$-module $P$ with an $\mathcal{N}\Gamma$-module homomorphism $\lambda\colon P\to P^*=\Hom_{\mathcal{N}\Gamma}(P,\mathcal{N}\Gamma)$ comprise a \emph{hermitian form} $(P,\lambda)$ if $\lambda(x)(y)=\ol{\lambda(y)(x)}$ for all $x,y\in P$. A hermitian form is \emph{nonsingular} if $\lambda$ has the property $\dim_{\mathcal{N}\Gamma}\ker(\lambda)=0=\dim_{\mathcal{N}\Gamma}\coker(\lambda)$. Using the $\mathcal{N}\Gamma$-dimension function it is then possible to define an \emph{$L^{(2)}$-signature} homomorphism from the Witt group $\sgn^{(2)}\colon L^0_p(\mathcal{N}\Gamma)\to \R$; see \cite[Section 3.1]{MR3257550}. When $P$ is moreover free, of rank $n$ (say), a hermitian form $(P,\lambda)$ may be described by some $U\in \Herm_{n\times n}(\mathcal{N}\Gamma)$, and the $L^{(2)}$-signature is given by
\[
\sgn^{(2)}(U)=\tr_\Gamma(p_+(U))-p_-(U))\in\R.
\]
Here $\tr_\Gamma\colon K_0(\mathcal{N}\Gamma)\to \C$ is the von Neumann trace of $\mathcal{N}\Gamma$ \cite[Definition 1.2]{Lueck:2002-1}. The operators $p_+,p_-\colon\Herm_{n\times n}(\mathcal{N}\Gamma)\to K_0(\mathcal{N}\Gamma)$ are defined by taking the characteristic functions of $(0,\infty)$ and $(-\infty,0)$ respectively, considering these as operators on the spectrum of a hermitian matrix, then applying the functional calculus; see~\cite[Definition 5.2]{MR1973052}, and the preceding discussion there.

\smallskip

Suppose there exists a connected, oriented, compact, topological $4\ell$-manifold $W$ cobounding~$r$ disjoint copies of $M$. Suppose further that there is an injective group homomorphism $j\colon G\hookrightarrow \Gamma$, where $\Gamma$ is also a discrete group, and a map $\psi$ making the following diagram of maps of spaces commute
\begin{equation}\label{eq:diagram}
\begin{tikzcd}
\bigsqcup_r M\ar{rr}{\bigsqcup_r\varphi}\ar{d}{i} && BG\ar{d}{Bj}\\
W\ar{rr}{\psi} && B\Gamma
\end{tikzcd}
\end{equation}

Using the inclusions $\Z\Gamma\subset \C\Gamma\subseteq\mathcal{N}\Gamma$, the representation $\psi$, and a cellular chain complex for $W$, define a chain complex $C(W;\mathcal{N}\Gamma):=\mathcal{N}\Gamma\otimes_{\Z\Gamma}C(W;\Z\Gamma)$ of free, finitely generated left modules over $\mathcal{N}\Gamma$, and from this define homology $\mathcal{N}\Gamma$-modules~$H_*(W;\mathcal{N}\Gamma)$. There is a $\mathcal{N}\Gamma$-coefficient hermitian intersection form\[\lambda\colon H_{2l}(W;\mathcal{N}\Gamma)\times H_{2l}(W;\mathcal{N}\Gamma)\to \mathcal{N}\Gamma.\]The form $(H_{2l}(W;\mathcal{N}\Gamma),\lambda)$ becomes a nonsingular form on a projective $\mathcal{N}\Gamma$-module after applying the $\mathcal{N}\Gamma$ projectivisation functor $\mathbf{P}$ (see \cite[\textsection 3.1]{MR3257550}), so that we may define $\sgn ^{(2)}(W,\psi):=\sgn^{(2)}(\mathbf{P}(H_{2l}(W;\mathcal{N}\Gamma)),\mathbf{P}(\lambda))\in\R$.

\begin{definition}\label{def:rho}Given $(M,\phi)$ as above and for any $(W,\Gamma,j,\psi)$ as above, the \emph{$L^{(2)}$ $\rho$-invariant} of $(M,\varphi)$ is defined to be
\begin{equation}\label{eq:rho}
\rho^{(2)}(M,\varphi):= \frac{1}{r}\left(\sgn ^{(2)}(W,\psi)-\sgn (W)\right)\in\R,
\end{equation}
where $\sgn (W)$ is the ordinary signature of $W$.
\end{definition}

Given $\varphi\colon M \to BG$, the required group homomorphism $j\colon G\hookrightarrow \Gamma$ and pair~$(W,\psi)$ always exist. Moreover the resulting $\rho^{(2)}(M,\varphi)\in\R$ is well-defined, independent of the choice $(W,\Gamma, j, \psi)$. We refer to \cite[\textsection 2.1]{MR3493628} for proof of both these facts (the discussion in the next remark recalls the development of Cha's proof).

\begin{remark} It is a result of Hausmann \cite[Theorem 5.1]{MR627098} that a closed, connected, oriented, topological manifold that bounds, must in fact bound some $W$ such that the inclusion induced map $\pi_1(M)\to \pi_1(W)$ is an injection. Chang and Weinberger \cite{MR1988288} used this result, together with Novikov additivity and $L^{(2)}$-induction, to show $\rho^{(2)}(M,\phi)$ exists and is well-defined in the case $\phi$ is the identity map for $\pi_1(M)$. In their paper, they also provided a different argument for Hausmann's result. Cha \cite[\textsection 2.1]{MR3493628} observed that this new Chang-Weinberger argument generalises to arbitrary $\phi$ and used this to prove~$\rho^{(2)}(M,\phi)$ exists and is well-defined in general.
\end{remark}

\subsection{Metabelian doubly slice obstruction from $L^{(2)}$ $\rho$-invariants}

For each prime~$p$, write~$\F_p$ for the field with~$p$ elements.

\begin{definition}Let $\mathcal{A}$ be a finitely generated left ${\Z[\Z]}$-module and let $\phi\colon G\to \Z\ltimes \mathcal{A}$ be a group homomorphism. Write
\[\phi^\Q\colon G\xrightarrow{\phi}\Z\ltimes\mathcal{A}\to \Z\ltimes\left(\Q[\Z]\otimes_{\Z[\Z]}\mathcal{A}\right),\]and for each prime $p\in\N$ write
\[\phi^{\F_p}\colon G\xrightarrow{\phi}\Z\ltimes\mathcal{A}\to \Z\ltimes\left(\F_p[\Z]\otimes_{\Z[\Z]}\mathcal{A}\right)\]for the homomorphisms given by post-composing $\phi$ with the respective tensor products. Define a set of representations \[\mathcal{H(\phi)}:=\left\{\phi^\Q,\phi^{\F_p}\mid \text{$p\in\N$ a prime}\right\}.\]
\end{definition}

\noindent The following is our main doubly slice obstruction.

\begin{theorem}\label{thm:main}Let $K$ be a $(4\ell-3)$-knot with metabelian group $\pi_K\cong \Z\ltimes H$. If~$K$ is doubly slice then then there exists a decomposition of ${\Z[\Z]}$-modules $H\cong A\oplus B$ such that for the homomorphisms
\begin{align*}
  \phi_A \colon \pi_K &\to \Z\ltimes A,\\
\phi_B\colon \pi_K &\to \Z\ltimes B,
\end{align*}
corresponding to the projections of $H$ to $A$ and $B$ respectively, we have that\[\rho^{(2)}(M_K,\phi)=0,\qquad\text{for all $\phi\in\mathcal{H}(\phi_A)\cup\mathcal{H}(\phi_B)$.}\]
\end{theorem}

\begin{proof}As in \fullref{prop:metabelian}, there is an isomorphism $H\cong H_1(X_K;{\Z[\Z]})\cong H_1(M_K;{\Z[\Z]})$. Let $D_A, D_B\subseteq D^{4\ell}$ be complementary slice discs for $K$ with respective disc exteriors $W_A, W_B$. The modules
\begin{align*}
A&:=\ker\big(H_1(M_K;{\Z[\Z]})\to H_1(W_A;{\Z[\Z]})\big)\\
B&:=\ker\big(H_1(M_K;{\Z[\Z]})\to H_1(W_B;{\Z[\Z]})\big)
\end{align*}
determine a decomposition of ${\Z[\Z]}$-modules $A\oplus B\cong H_1(M_K;{\Z[\Z]})\cong H$, since $W_A \cup_{X_K} W_B$ is the exterior of an unknot, so $H_i(W_A \cup_{X_K} W_B;\Z[\Z]) =0$ for $i=1,2$.

Fix $\phi\in \mathcal{H}(\phi_A)\cup\mathcal{H}(\phi_B)$. We will assume that $\phi\in\mathcal{H}(\phi_A)$, the case of $\phi\in\mathcal{H}(\phi_B)$ can be argued entirely similarly. We have $\phi=\phi_A^\F$ for one of $\F=\Q$ or $\F=\F_p$, for some prime $p\in\N$. Write $j_B\colon M_K\hookrightarrow W_B$ for the inclusion of the boundary. The inclusions of the boundary determine an isomorphism $H_1(M_K;{\Z[\Z]})\cong H_1(W_A;{\Z[\Z]})\oplus H_1(W_B;{\Z[\Z]})$. This implies there is an isomorphism $A\cong H_1(W_B;{\Z[\Z]})$ such that, under the isomorphism $A\oplus B\cong H_1(M_K;{\Z[\Z]})\cong H$, the map $(j_B)_*$ corresponds to the projection of $H$ to $A$. In particular, the sequence of maps \[\psi\colon \pi_1(W_B)\to \pi_1(W)/\pi_1(W)^{(2)}\cong \Z\ltimes H_1(W_B;{\Z[\Z]})\cong \Z \ltimes A \to\Z\ltimes (\F[\Z]\otimes_{\Z[\Z]} A)\]determines an extension of the representation $\phi$ to $\pi_1(W_B)$.

Define $\Gamma:=\Z\ltimes \left( \F[\Z]\otimes_{\Z[\Z]} A\right)$. We wish to use $(W_B,\psi)$ to calculate $\rho^{(2)}(M_K,\phi)=0$. First, note that as $H_*(W_B;\Q)\cong H_*(S^1;\Q)$ we have that $\sgn(W_B)=0$, so if we can show that $\sgn^{(2)}(W_B,\psi)=0$, then the proof will be complete. In fact we will show the stronger statement that $\dim_{\mathcal{N}\Gamma}H_{2l}(W_B;\mathcal{N}\Gamma)=0$.

Recall a locally compact topological group $G$ is \emph{amenable} if $G$ admits a finitely-additive measure which is invariant under the left multiplication; see e.g.\ \cite{Paterson:1988-1}. We now recall the theorem of Cha-Orr \cite[Theorem 6.6]{Cha-Orr:2009-01} which states that if $G$ is an amenable group in the Strebel class $D(R)$, for some $R$, and $C$ is a bounded chain complex of finitely generated left $\Z[G]$-modules, then when $H_i(R\otimes_{\Z[G]}C)=0$ for $i\leq n$ we also have $\dim_{\mathcal{N}G}H_*(\mathcal{N} G\otimes_{\Z[G]} C)=0$ for $i\leq n$.

As $\F$ is a field, the group $\Gamma$ is PTFA as the normal subgroup $\F[\Z]\otimes_{\Z[\Z]} A$ is already torsion free abelian. It is well known that all solvable groups, a class which includes PTFA groups, are amenable, and we saw in \fullref{lem:strebelclass} that all PTFA groups are in~$D(\F)$. To apply the theorem of Cha-Orr, we will set $R=\F$, $G=\Gamma$, and must choose an appropriate chain complex $C$. Let $S^1\hookrightarrow W_B$ be an inclusion representing a generator of $H_1(W_B;\Z)\cong\Z$, so in particular the obvious map $\pi_1(S^1)\cong\Z\to \Z\ltimes A$ is compatible with $\psi$ and we may take $\Z\Gamma$-coefficient homology for this $S^1$ compatibly with $W_B$. It now makes sense to define~$C:=\cone(C(S^1;\Z\Gamma)\to C(W_B;\Z\Gamma))$, the algebraic mapping cone. A straightforward Mayer-Vietoris argument using the decomposition $D^{4\ell} = W_B \cup_{D^{4\ell-2} \times S^1} D^{4\ell-2} \times D^2$ shows that $H_*(W_B,S^1;\Z)=0$, so in particular $H_*(\F\otimes_{\Z\Gamma}C)=0$ and hence $\dim_{\mathcal{N}\Gamma}H_i(\mathcal{N} \Gamma\otimes_{\Z \Gamma} C)=0$ for all $i\in\Z$ by the Cha-Orr theorem.

Denote  the chain complex desuspension functor by $\Sigma^{-1}$. The short exact sequence of cellular $\Z\Gamma$-module chain complexes\[0\to C(W_B;\Z\Gamma)\to \cone(C(S^1;\Z\Gamma)\to C(W_B;\Z\Gamma))\to \Sigma^{-1}C(S^1;\Z\Gamma)\to 0\]is split exact in every homological degree (by construction), so applying the functor $\mathcal{N}\Gamma\otimes_{\Z\Gamma}-$ results in another short exact sequence and hence there is an induced long exact sequence of $\mathcal{N}\Gamma$-modules\[\dots\to H_{i-1}(\mathcal{N}\Gamma\otimes_{\Z\Gamma}C)\to H_i(S^1;\mathcal{N}\Gamma)\to H_i(W_B;\mathcal{N}\Gamma)\to H_{i}(\mathcal{N}\Gamma\otimes_{\Z\Gamma}C)\to\dots\]
As $\dim_{\mathcal{N}\Gamma}H_i(\mathcal{N} \Gamma\otimes_{\Z \Gamma} C)=0$ for all $i\in\Z$, this shows that for all $i\in\Z$ we have that $\dim_{\mathcal{N}\Gamma} H_i(S^1;\mathcal{N}\Gamma)= \dim_{\mathcal{N}\Gamma}H_i(W_B;\mathcal{N}\Gamma)$; see \cite[Theorem 6.7(4b)]{Lueck:2002-1}. But the chain complex $C(S^1;\Z\Gamma)$ is concentrated in degrees 0 and 1, so $H_{2l}(S^1;\mathcal{N}\Gamma)=0$ and hence we obtain $\dim_{\mathcal{N}\Gamma}H_{2l}(W_B;\mathcal{N}\Gamma)=\dim_{\mathcal{N}\Gamma}H_{2l}(S^1;\mathcal{N}\Gamma)=0$ as claimed.\end{proof}

\subsection{Calculation of $L^{(2)}$ $\rho$-invariants}

The main method for calculating $L^{(2)}$ $\rho$-invariants is to identify a cyclic subgroup $\Gamma'\subset \Gamma$ and use the following consequence of \emph{$L^{(2)}$ induction}; see \cite[Theorem 6.29]{Lueck:2002-1}.

\begin{proposition}\label{prop:trick}If $\Phi\colon\Gamma'\to\Gamma$ is an injective homomorphism, then for each $n\in\N$, there is a commutative diagram
\[
\begin{tikzcd}
\Herm_{n\times n}(\mathcal{N}\Gamma')\arrow[rr,]\arrow[dr, "\sgn^{(2)}"' ]&&\Herm_{n\times n}(\mathcal{N}\Gamma)\ar{ld}{\sgn^{(2)}}\\
&\R&
\end{tikzcd}
\]
\end{proposition}

\begin{proof} The commutative diagram of \cite[Proposition 5.13]{MR1973052} (below, left) induces a commutative diagram in $K$-theory (below, right).
\[
\begin{tikzcd}
\mathcal{N}\Gamma'\arrow[rr,]\arrow[dr, "\tr_{\Gamma'}"' ]&&\mathcal{N}\Gamma\ar{ld}{\tr_\Gamma}
&&
K_0(\mathcal{N}\Gamma')\arrow[rr,]\arrow[dr, "\tr_{\Gamma'}"' ]&&K_0(\mathcal{N}\Gamma)\ar{ld}{\tr_\Gamma}\\
&\C&
&&
&\C&
\end{tikzcd}
\]
The proposition now follows immediately from the description of the $L^{(2)}$-signature of a hermitian matrix as a trace $\sgn^{(2)}(U)=\tr_\Gamma(p_+(U))-p_-(U))$.
\end{proof}

This proposition provides a computational strategy first exploited in \cite{MR1973052}. If we can identify a cyclic subgroup, then we can make our calculations using coefficients in this restricted setting. When working over a cyclic group, the computation of $L^{(2)}$ signatures, and thus $\rho$-invariants, is very well understood in terms of ordinary signatures.

Let $\Gamma$ be a group and fix $g\in \Gamma$, with order $\ord(g)$. Let $\Gamma'=\langle u\,|\,u^{\ord(g)}=1\rangle$ be a cyclic group of order $\ord(g)$ and define an injective homomorphism $\Phi_g\colon\Gamma'\to \Gamma$ by $\Phi_g(u) = g$. For each $n\in\N$, there is an induced map $\Phi_g\colon \Herm_{n\times n}(\Z[\Gamma'])\to \Herm_{n\times n}(\Z [\Gamma])$. Let $\omega\in S^1\subset \C$. We define a further map
\[
\epsilon_\omega\colon\im(\Phi_g\colon \Herm_{n\times n}(\Z[\Gamma'])\to \Herm_{n\times n}(\Z [\Gamma]))\to \Herm_{n\times n}(\C);\quad g\mapsto \omega.
\]

\noindent The following proposition is essentially due to Cochran, Orr, and Teichner~\cite[\textsection 5]{MR1973052}.

\begin{proposition}\label{prop:Friedl} Let $M$ be a closed, oriented $(4\ell-1)$-manifold and $\phi\colon M\to BG$ for some group $G$. Suppose we have chosen $(W,\Gamma, j,\psi)$ as in Diagram (\ref{eq:diagram}), with the further property that $\sgn(W)=0$. Fix $g\in \Gamma$ and suppose that the $\psi$-twisted intersection form of $W$ is represented by a matrix $U\in\im(\Phi_g)\subset \Herm_{n\times n}(\Z[\Gamma])$.
\begin{enumerate}[leftmargin=*]

\item \label{item:case1} Suppose $g$ has finite order $k\in\N$. Then
\[
\rho^{(2)}(M,\varphi)=\frac{1}{rk}\sum_{j=1}^k\sgn(\eps_{\omega^j}(U)),
\]
where $\omega\in S^1\subset \C$ is a primitive $k$th root of unity.
\item \label{item:case2}
Suppose $g$ has infinite order. Then
\[
\rho^{(2)}(M,\varphi)=\frac{1}{2\pi r}\int_{\omega\in S^1}\sgn(\eps_{\omega}(U)).
\]

\end{enumerate}
\end{proposition}

\begin{proof}\leavevmode
\begin{enumerate}[leftmargin=*]
\item Write $\Gamma'=C_k=\langle u\,|\,u^k=1\rangle$ for the cyclic group of order $k$. Write $V\in\Herm_{n\times n}(\Z[C_k])$ for the matrix such that $\Phi_g(V)=U$. By \fullref{prop:trick}, the signature $\sgn^{(2)}(W,\psi)=\sgn^{(2)}(U)$ is given by $\sgn^{(2)}(V)$. Using as basis $1,u,\dots,u^{k-1}$, we obtain an isomorphism $\C[C_k]\cong \C^k$. For clarity, we will write $V$ as $V^\C$ when we view it as an automorphism of $(\C^{k})^n$ via the left regular representation. Because $C_k$ is a finite group, there is equality $\C[C_k]=\mathcal{N} C_k$, and the von Neumann trace of a general element $a=a_0+a_1u+\dots a_{k-1}u^{k-1}\in\C[C_k]$ is $\tr_{\mathcal{N}\Gamma}(a)=a_0\in\C$. On the other hand, recall that the standard trace of $a\in\C[C_k]$ under the regular representation is given by $\tr(a^\C)=k\cdot a_0$, and that ordinary signature $\sgn(V^\C)$ may be interpreted as an ordinary trace $\tr\colon K_0(\C\Gamma)\to\Z$. Comparing the two types of trace applied to $p_+(V)-p_-(V)\in K_0(\mathcal{N}\Gamma)= K_0(\C\Gamma)$, we obtain $k\cdot \sgn^{(2)}(V)=\sgn(V^\C)$. It remains to calculate $\sgn(V^\C)$, but this is standard.

Decompose $(\C[C_k])^n$ according to \fullref{lem:obvs}, noting that the irreducible representations of $C_k$ are given by the one-dimensional characters $\chi^j(u)=\omega^j\in \C$, $j=1,\dots,k$. The signature of $V^\C$ restricted to the $\omega^j$-eigenspace is given by $\sgn(\epsilon_{\omega^j}(U))$, and hence
\[
\sgn^{(2)}(V)=\frac{1}{k}\sgn(V^\C)=\frac{1}{k}\sum_{j=1}^k\sgn(\eps_{\omega^j}(U)).
\]
The claimed result follows, as $\sgn(W)=0$ by assumption.

\item Write $\Gamma'=\Z\langle u\rangle$. Write $V\in\Herm_{n\times n}(\Z[u,u^{-1}])$ for the matrix such that $\Phi_g(V)=U$. By \fullref{prop:trick}, the signature $\sgn^{(2)}(W,\psi)$ is given by $\sgn^{(2)}(V)$. By \cite[Lemma 4.5]{MR1973052} this may be calculated as
\[
\sgn^{(2)}(V)=\frac{1}{2\pi}\int_{\omega\in S^1}\sgn(\eps_{\omega}(U)).
\]
The claimed result follows, as $\sgn(W)=0$ by assumption.
\end{enumerate}
\end{proof}

\section{New non doubly slice knots}

\begin{definition} Let $G$ be a group. We will call a matrix $U\in\Herm_{r\times r}(\Z[G])$ \emph{even} if it is of the form $U=V+\overline{V}^T$ for some matrix $V$ with entries in $\Z[G]$.
\end{definition}

For a group homomorphism $\alpha\colon G\to H$, we abuse notation and also use $\alpha$ to denote the induced homomorphisms $\Z[G]\to \Z[H]$ and $\Herm_{r\times r}(\Z[G])\to \Herm_{r\times r}(\Z[H])$.

\begin{definition} Let $\alpha\colon G\to H$ be a group homomorphism. Then $U\in\Herm_{r\times r}(\Z[G])$ is \emph{$\alpha$-nonsingular} if $\alpha(U)$ is nonsingular. In the particular case that $K$ is a knot and $\alpha\colon \pi_K\to \Z$ is the abelianisation, we will call an $\alpha$-nonsingular matrix \emph{$\Z[\Z]$-nonsingular}.
\end{definition}

An even, $\alpha$-nonsingular matrix represents an element of the Cappell-Shaneson homology surgery obstruction group $\Gamma_{4\ell}(\alpha\colon \Z[G]\to \Z[H])$ \cite{Cappell-Shaneson:1974-1}. We recall the standard procedure for realising such elements as the surgery obstruction associated to a $4\ell$-manifold with boundary.

\begin{theorem}[{\cite[Theorem 1.8 \& Addendum 1.8]{Cappell-Shaneson:1974-1}}]\label{thm:CapSha}
Let $\ell>1$, and let $(X,\partial X)$ be a $(4\ell-1)$-manifold with boundary. Suppose there is a group homomorphism $\phi\colon\pi\to G$, where $\pi=\pi_1(X)$. Let $U\in\Herm_{r\times r}(\Z[G])$ be even. Then there exists a $4\ell$-dimensional manifold triad $(W;\partial_+W,\partial_-W)$ such that $\partial_+W=-X\sqcup Y$, $\partial_-W=\partial X\times[0,1]$ $($implying $\partial Y=\partial X)$, and such that:
\begin{itemize}
\item There is a degree one map $(f, \Id)\colon (Y,\partial Y)\to (X,\partial X)$ such that $\pi_i(f)$ is an isomorphism for $0\leq i\leq 2\ell-2$.
\item There is a degree one map $F\colon W\to X\times[0,1]$ such that $\pi_i(F)$ is an isomorphism for $0\leq i\leq 2\ell-1$, and the middle dimensional homology $H_{2\ell}(W;\Z[G])$ is a free module and the middle dimensional intersection form with coefficients in $\Z[G]$ is given by ~$U$.
\item If $U$ is $\alpha$-nonsingular for some $\alpha\colon G\to H$, the map $f$ induces an isomorphism on homology with $\Z[H]$-coefficients.
\end{itemize}
\end{theorem}

\noindent For our purposes, this has the following corollary.

\begin{corollary}\label{cor:newknots} For $\ell>1$, let $K$ be a $(4\ell-3)$-knot. Let $\phi\colon \pi_K\to G$ be a surjective homomorphism factoring through the abelianisation of $\pi_K$ and let $U\in\Herm_{r\times r}(\Z[G])$ be even. Then there exists a $(4\ell-3)$-knot $K'$ with $\pi_{K'}=\pi_K$ and such that $M_K$ and $M_{K'}$ are cobordant, via a $4\ell$-manifold $Z$, where $Z$ has $\Z[G]$-coefficient intersection form $U$.

If $U$ is moreover $\Z[\Z]$-nonsingular, then there is a $\Z[\Z]$-homology equivalence between the exterior of $K$ and the exterior of $K'$, that is the identity on the boundary and preserves meridians.
\end{corollary}
\begin{proof}

Apply \fullref{thm:CapSha} to $(X_K,\partial X_K=S^{4\ell-3}\times S^1)$, to obtain $(W;\partial_+W,\partial_-W)$ with interior boundary $\partial_-W=S^{4\ell-3}\times S^1\times[0,1]$ and exterior boundary $\partial_+W=-X_K\sqcup Y$. To confirm that $Y$ is the exterior of some knot $K'$, glue a copy of $S^{4\ell-3}\times D^2$ along $\partial Y$. The resulting manifold is homotopy equivalent to $S^{4\ell-1}$, and thus homeomorphic to $S^{4\ell-1}$ by the (topological) Poincar\'{e} conjecture. The glued-in core $S^{4\ell-3}\times \{\pt\}=:K'$ is the promised knot.

Now construct a $4\ell$-manifold with boundary $(Z, -M_K\sqcup M_{K'})$ by
\[
Z:=W\cup_{\partial_-W}\left(D^{4\ell-2}\times S^1\times[0,1]\right).
\]
This glueing is unambiguous, as every homeomorphism of $S^{4\ell-3}\times S^1\times[0,1]$ extends to a homeomorphism of $D^{4\ell-2}\times S^1\times[0,1]$.
\end{proof}

\begin{remark}
  We could change the construction so that we have a diffeomorphism to $S^{4\ell-3}$, just by connect summing with the appropriate exotic sphere away from the knot.  This will not change any $L^{(2)}$-signature obstruction, since that is a homeomorphism invariant.
\end{remark}

\begin{proof}[Proof of \fullref{thm:family}]
Fix $n>3$ odd and let $H=\Z[t,t^{-1}]/p(t)$ where $p(t)\in\Z[t,t^{-1}]$ is irreducible. Use \fullref{lem:buildtheknot} to build a doubly slice $n$-knot $K$ with metabelian fundamental group $\pi_K\cong\Z\ltimes H$. By \fullref{thm:main}, we obtain a decomposition of ${\Z[t,t^{-1}]}$-modules $H\cong A\oplus B$. As $p$ is irreducible we may assume that $B=0$ and $A=H$. Denote by $\alpha\colon \pi_K\to\pi_K$ the identity map.

Regarding $H$ as an abelian group, choose an element of infinite order $g\in H$ and write $\Phi_g\colon\Z\langle u\rangle\to \pi_K$ for the injective group homomorphism defined by $u\mapsto g$. For every $r\in\N$, this extends to a homomorphism $\Phi_g\colon\Herm_{r\times r}(\Z[u,u^{-1}])\to\Herm_{r\times r}(\Z[\pi_K])$. Suppose we have chosen an even matrix $U\in \im(\Phi_g)$ such that $U$ becomes nonsingular upon applying the abelianisation induced map $\Z[\pi_K]\to\Z[t,t^{-1}]$, and moreover has vanishing signature under the augmentation $t\mapsto 1$ (we will choose such a $U$ later).
Apply \fullref{cor:newknots} to $U$ and $K$,
to obtain a new knot $K'$ and the $4\ell$-dimensional manifold $Z$, with intersection form $U$. As $U$ was $\Z[\Z]$-nonsingular, there is a $\Z[\Z]$-homology equivalence between $X_K$ and $X_{K'}$ that is the identity on the boundary and preserves meridians. We note that $\pi_{K'}=\Z\ltimes H$.

Applying \fullref{prop:Friedl}, we may calculate that
\begin{equation}\label{eq:wall}
\rho^{(2)}(M_{K'},\alpha^\Q)-\rho^{(2)}(M_{K},\alpha^\Q)=\dfrac{1}{2\pi}\int_{w\in S^1}\sgn(\eps_\omega(U)),
\end{equation}
where, recall,
\[
\eps_\omega\colon \im(\Phi_g\colon \Herm_{r\times r}(\Z[u,u^{-1}])\to \Herm_{r\times r}(\Z[\pi_K]))\to \Herm_{r\times r}(\C);\qquad g\mapsto \omega.
\]
As $K$ is doubly slice, by \fullref{thm:main}, we have $\rho^{(2)}(M_{K},\alpha^\Q)=0$.

Our next task is to choose $U$ with the properties specified above and so that the integral in \fullref{eq:wall} in nonzero. A second application of \fullref{thm:main} will then complete the proof that $K'$ is not doubly slice. We argue just as in \cite[\textsection 5]{MR1973052} that the matrix
\[
U:=\left(\begin{matrix} g-2 +g^{-1} & 1\\ 1 & g-2 +g^{-1}\end{matrix}\right)\in\Herm_{2\times 2}(\Z[\pi_K])
\]
works. First, $U=V+\overline{V}^T$, where $V$ is the matrix
\[
V:=\left(\begin{matrix} g-1 & 1\\ 0 & g-1\end{matrix}\right)\in \operatorname{Mat}_{2\times 2}(\Z[\pi_K]),
\]
so $U$ is even.
Next, because $g$ lies in $\pi_K^{(1)}$, $g$ is sent to $1$ under the abelianisation $\pi_K\to\Z\langle t\rangle$ and the matrix becomes the standard hyperbolic matrix, confirming $U$ is $\Z[\Z]$-nonsingular. The further map induced by augmentation $t\mapsto 1$, does not change the matrix, so the augmentation matrix has vanishing signature as required. It remains to calculate that the integral is nonzero. We have
\[
\det(\eps_\omega(U))= (\omega-2 +\omega^{-1})^2-1=(\omega-1 +\omega^{-1})(\omega-3 +\omega^{-1}),
\]
which is nonzero for $\omega\in S^1$ except at the primitive sixth roots of unity. The circle $S^1$ is separated into two arcs with boundary these roots of unity, and the signature $\sgn(\eps_\omega(U))$ is constant on each arc. The reader may easily check that $\sgn(\eps_1(U))=0$ and $\sgn(\eps_{-1}(U))=-2$. Thus
\[
\frac{1}{2\pi}\int_{\omega\in S^1}\sgn(\eps_\omega(U))=\frac{2}{3}\cdot (-2)+\frac{1}{3}\cdot 0=-\frac{4}{3}\neq 0.
\]
It follows that $K'$ is not doubly slice.

We next produce an infinite family of different knots using the construction above. For this, denote by $\Phi_m$ the $m$th cyclotomic polynomial and recall the well-known property of cyclotomic polynomials that if $m\neq 1$ and $m$ is not a prime power then $\Phi_m(1)=1$. So $\Phi_m(t)$ is an irreducible Alexander polynomial. Denote by $K(U,m)$ the non doubly slice knot obtained by applying the construction above to the matrix $U$ and the cyclotomic polynomial $m$ when $m\neq 1$ is not a prime power. Define
\[
\mathcal{F}=\{K(U,m)\mid \text{$m$ is divisible by at least 3 distinct primes}\}.
\]
This is the infinite family we seek. As noted above, we must have that the knot group of $K(U,m)$ is $\Z\ltimes \left(\Z[\Z]/\Phi_m\right)$, which shows that the knots in the family $\mathcal{F}$ are pairwise distinct.

Finally, we will show that Ruberman's Casson-Gordon invariants vanish for the knots in the family $\mathcal{F}$. The first stage in defining Ruberman's obstructions \cite[Theorem 2.2]{Ruberman-doubly-slice-I} for an $(4\ell-3)$-knot $K$ to be doubly slice is to first pass to some $r$-fold cyclic branched cover $\Sigma_r(K)$, where $r$ is a prime power (this is the ``$M$'' in the notation of \cite[Theorem 2.2]{Ruberman-doubly-slice-I}). Next one must find a homomorphism from $H_1(\Sigma_r(K))$ to a finite cyclic group $\Z/d\Z$. As the Alexander polynomials of the knots in $\mathcal{F}$ are, by construction, cyclotomic of order divisible by at least 3 distinct primes, we may apply \fullref{thm:livingston}. That shows that any homomorphism $H_1(\Sigma_r(K))\to\Z/d\Z$ is trivial for knots in $\mathcal{F}$. As a consequence,  Ruberman's invariants vanish for the knots in $\mathcal{F}$.
\end{proof}

%

\section{Abelian obstructions vanish on our knots}\label{section:first-order-obstructions}

In this section we recall the various notions of algebraic sliceness and algebraic double sliceness that have been defined using abelian coefficient systems. All the non doubly slice knots we constructed for \fullref{thm:family} will be algebraically doubly slice in the strongest sense we describe below. In~\cite[Proposition 4.11]{MR3604378}, the first author gave a strong version of algebraic double sliceness using the entire chain complex of the knot exterior as an obstruction. Another such condition, using the full cohomology ring of the knot exterior, was given in~\cite[Theorem A]{Levine-doubly-slice}. We will define a single strong algebraic double sliceness condition that implies all other previous versions and that is satisfied by all the non doubly slice knots in our family $\mathcal{F}$.

\subsection{Strong algebraic slices}

We briefly recall some elements of Ranicki's Algebraic Theory of Surgery \cite{MR560997,MR566491}, suppressing many details and highlighting only the relevant features for our purposes. Suppose $C_*$ is a finitely generated $\Z[\Z]$-module chain complex such that $C_r=0$ outside of the range $0\leq r\leq n$. An $n$-dimensional \emph{symmetric structure} is an equivalence class of collections of (higher) chain maps
\[
\phi=\{\phi_s\in\Hom_{\Z[\Z]}(C^{n-r+s},C_r)\mid r\in\Z,s\geq 0\},
\]
satisfying certain interrelations \cite[p.~7]{MR620795}. A \emph{symmetric complex} $(C,\phi)$ is called \emph{Poincar\'{e}} if the chain map
\[
\phi_0\colon C^{n-r}\to C_r,\qquad r\geq 0,
\]
  is a chain homotopy equivalence. Suppose $D$ is a finitely generated $\Z[\Z]$-module chain complex such that $D_r=0$ outside of the range $0\leq r\leq n+1$. Given a chain map $f\colon C\to D$, an $(n+1)$-dimensional \emph{symmetric structure} is an equivalence class of collections
\[
(\delta\phi,\phi)=\{(\delta\phi_s,\phi_s)\in\Hom_{\Z[\Z]}(D^{n+1-r+s},D_r)\oplus\Hom_{\Z[\Z]}(C^{n-r+s},C_r)\mid r\in\Z, s\geq 0\},
\]
satisfying certain interrelations \cite[p.~15]{MR620795} (that depend on $f$). A \emph{symmetric pair} $(f\colon C\to D,(\delta\phi,\phi))$ is called \emph{Poincar\'{e}} if the chain map
\[
(\delta\phi_0\,\,\phi_0)\colon \cone  (f)^{n+1-r}=D^{n+1-r}\oplus C^{n-r}\to D_r,\qquad r\geq 0
\]
is a chain homotopy equivalence. Two symmetric Poincar\'{e} pairs $(f_\pm\colon C\to D_\pm,(\delta_\pm\phi,\phi))$, indicated by `$+$' and `$-$',  are \emph{complementary} if
\[
\left(\begin{matrix} f_+\\f_-\end{matrix}\right)\colon C\to D_+\oplus D_-
\]
is a chain homotopy equivalence.

The \emph{Blanchfield complex} $(C_K,\phi_K)$ of an $n$-knot $K$ is constructed in~\cite[Chapter 7.8]{MR620795} using Ranicki's \emph{symmetric construction} \cite[Proposition 1.2]{MR566491}, and developed further in~\cite[Section 4]{MR3604378}. Briefly, $(C_K,\phi_K)$ is an $(n+2)$-dimensional symmetric Poincar\'{e} complex, such that the isomorphism class of $(C_K,\phi_K)$ is an invariant of the isotopy class of $K$. By construction there is a decomposition
\[
C_K\oplus C_*(\overline{D^{n+1}\times S^1})\simeq C_*(\overline{X_K}).
\]
So in particular there is an isomorphism in reduced homology $H_*(C_K;\Z[\Z])\cong \widetilde{H}_*(X_K;\Z[\Z])\cong \widetilde{H}_*(M_K;\Z[\Z])$, and similarly reduced cohomology. From this $C_K$ is seen to be $S$-acyclic.

\begin{definition}\label{def:strongalgslice} A \emph{strong algebraic slice} for an $n$-knot $K$ is a symmetric Poincar\'{e} pair $(f\colon C_K\to D,(\delta\phi_K,\phi_K))$, such that the image of the induced morphism on cohomology $f^*\colon H^*(D;\Z[\Z])\to H^*(C_K;\Z[\Z])$ is closed under cup product.

A knot that admits a strong algebraic slice is called \emph{strongly algebraically slice} and a knot that admits two  strong algebraic slices that are complementary as symmetric Poincar\'{e} pairs is called \emph{strongly algebraically doubly slice}.
\end{definition}

\begin{proposition}\label{prop:obvs}
A slice disc for $K$ determines a strong algebraic slice for $K$. Complementary slice discs for $K$ determine complementary strong algebraic slices for $K$.
\end{proposition}

\begin{proof} Let $D$ be a slice disc for $K$. The slice disc exterior $W:=D^{n+3}\sm \nu D$ has boundary the surgery manifold $M_K$. We refer the reader to~\cite[Proposition 4.11]{MR3604378} for the proof that $(W,M_K)$ determines an algebraic nullbordism for $(C_K,\phi_K)$ as required. The condition on the cup product is also immediate from the construction in that proof and naturality of the cup product under inclusion.
\end{proof}

\subsection{The Blanchfield pairing}

For the reader's convenience we recall a standard description of the Blanchfield pairing. The description is based on~\cite[\textsection 1--4]{Levine-77-knot-modules}.

\begin{proposition}\label{prop:blanchfieldwithcxs}Suppose that $f\colon C\to D$ is a map of finitely generated $\Z[\Z]$-module chain complexes that are $S$-acyclic. Suppose that $(f\colon C\to D,(\delta\phi,\phi))$ is an $(m+1)$-dimensional Poincar\'{e} pair. Then for $1<i<m$ there is a nonsingular pairing\[Bl \colon FH^{i}(D;{\Z[\Z]})\times FH^{m-i+1}(D,C;{\Z[\Z]})\to {\Q(\Z)}/{\Z[\Z]}.\]
\end{proposition}

The proof of \fullref{prop:blanchfieldwithcxs} is based on the following proposition from \cite[\textsection 2]{Levine-77-knot-modules}.

\begin{proposition}[Levine]\label{prop:extsequence}Suppose $R$ and $S$ are rings with unit, and $R$ has homological dimension 2 $($that is, any finitely generated module over $R$ has a length~2 resolution by projective $R$-modules$)$. Suppose $X$ is a projective left $R$-module chain complex  and $G$ is an $(R,S)$-bimodule. If $\Hom_R(H_p(X;R),G)=0$ for all $p\in\Z$, then for all $r\in\Z$ there is a short exact sequence of right $S$-modules \[0\to\Ext^2_R(H_{r-2}(X;R),G)\to H^r(X;G)\to \Ext^1_R(H_{r-1}(X;R),G)\to 0.\]
\end{proposition}

\begin{proof}[Proof of \fullref{prop:blanchfieldwithcxs}] To prove the proposition, we will describe a sequence of four right ${\Z[\Z]}$-module morphisms and show they are all isomorphisms. The pairing will then be defined as the adjoint of the composed maps in the sequence. This description will also show that the pairing is nonsingular.

We will describe and analyse the maps in the sequence:
\[
FH^{i}(D;{\Z[\Z]})\xrightarrow{UCT}\Ext^1_{\Z[\Z]}(H_{i-1}(D;{\Z[\Z]}),{\Z[\Z]})\cong\hspace{4cm}
\]
\[
\qquad\qquad \Ext^1_{\Z[\Z]}(FH_{i-1}(D;{\Z[\Z]}),{\Z[\Z]})\xleftarrow{\cong}\Hom_{\Z[\Z]}(FH_{i-1}(D;{\Z[\Z]}),{\Q(\Z)}/{\Z[\Z]})
\]
\[
\hspace{2cm}\xleftarrow{PD} \Hom_{\Z[\Z]}(FH^{m-i+1}(D,C;{\Z[\Z]}),{\Q(\Z)}/{\Z[\Z]}).
\]

For the first map, note ${H}_*(D;{\Q(\Z)})=0$ implies that $\Hom_{\Z[\Z]}(H_p(D;\Z[\Z]),\Z[\Z])=0$ for all $p\in\Z$. Then as ${\Z[\Z]}$ has homological dimension 2 we may apply  \fullref{prop:extsequence} to the case that $R={\Z[\Z]}$, $X=D$, and $G={\Z[\Z]}$ considered as a $({\Z[\Z]},{\Z[\Z]})$-bimodule. This results in exact sequences of right ${\Z[\Z]}$-modules for all $r\in\Z$:
\begin{align*}
0\to \Ext^2_{\Z[\Z]}({H}_{r-2}(D;\Z[\Z]),{\Z[\Z]})&\to {H}^r(D;{\Z[\Z]})\\
&\to \Ext^1_{\Z[\Z]}({H}_{r-1}(D;{\Z[\Z]}),{\Z[\Z]})\to0.
\end{align*}
For a general ${\Z[\Z]}$-module $A$, $F\Ext^2_{\Z[\Z]}(A,{\Z[\Z]})=0$ and $F\Ext^1_{\Z[\Z]}(A,{\Z[\Z]})=\Ext^1_{\Z[\Z]}(A,{\Z[\Z]})$; see \cite[3.2, 3.3]{Levine-77-knot-modules}. Hence by applying the functor $F$ to our exact sequence, and setting $r=i$, we obtain the first of our claimed isomorphisms in the sequence, called $UCT$.

The second isomorphism in the sequence follows from another general fact about ${\Z[\Z]}$-modules $A$, that $\Ext^1_{\Z[\Z]}(A,{\Z[\Z]})=\Ext^1_{\Z[\Z]}(FA,{\Z[\Z]})$.

The fourth isomorphism in our sequence follows from considering the long exact $\Ext$ sequence associated to the change of rings exact sequence $0\to {\Z[\Z]}\to {\Q(\Z)}\to {\Q(\Z)}/{\Z[\Z]}\to 0$, applied to a ${\Z[\Z]}$-module $A$:
\begin{align*}
\dots\to\Hom_{\Z[\Z]}(A,{\Q(\Z)})&\to \Hom_{\Z[\Z]}(A,{\Q(\Z)}/{\Z[\Z]})\\
&\to\Ext^1_{\Z[\Z]}(A,{\Z[\Z]})\to\Ext^1_{\Z[\Z]}(A;{\Q(\Z)})\to \cdots
\end{align*}
But setting $A=FH_{i-1}(Y,\partial_1 Y)$, the outer two terms vanish because $A$ is ${\Z[\Z]}$-torsion. The central map is then the next isomorphism we seek.

The final map is induced by using the $0$th chain map in the collection $(\theta,\phi\oplus-\phi)$. Namely it is the Poincar\'{e}-Lefschetz duality chain map $\left(\begin{smallmatrix}\delta\phi_0\\\pm \phi_0 f^*\end{smallmatrix}\right)\colon D^{n+1-*}\to\cone  (f)_*$.
This is an isomorphism from cohomology to homology by hypothesis, as the pair $(f\colon C\to D,(\delta\phi,\phi))$ was assumed to be Poincar\'{e}. Thus, when we apply the functor $F$ it is still an isomorphism, and the result follows.
\end{proof}

\begin{corollary}\label{cor:blanchfield}Suppose $(Y,\partial Y)$ is an oriented $m$-dimensional manifold with $($possibly empty$)$ boundary, together with a homomorphism $\pi_1(Y)\to\Z$ defining homology with ${\Z[\Z]}$ coefficients. Suppose further that the reduced homology, $\widetilde{H}_i(Y;{\Z[\Z]})$, $\widetilde{H}_i(\partial Y;{\Z[\Z]})$ is ${\Z[\Z]}$-torsion, and that multiplication by $1-t$ is an isomorphism on these modules. Then for $0<i<m-1$ there is a nonsingular pairing\[Bl \colon FH^{i}(Y;{\Z[\Z]})\times FH^{m-i+1}(Y,\partial Y;{\Z[\Z]})\to {\Q(\Z)}/{\Z[\Z]}.\]
\end{corollary}

\begin{proof}
The homomorphism $\pi_1(Y)\to\Z$, determines a map $Y\to K(\Z,1)=S^1$. This $S^1$ may be thickened so that there is a commutative diagram of spaces
\[
\begin{tikzcd}
\partial Y\arrow[r,"f"]\ar[d]&Y\ar{d}\\
S^1\times S^{m-2}\arrow[r]&S^1\times D^m
\end{tikzcd}
\]
where the vertical maps may be assumed to be degree one. Apply Ranicki's symmetric construction to the diagram to obtain $C:=\widetilde{C}_*(\partial Y;\Z[\Z])$ and $D:=\widetilde{C}_*(Y;\Z[\Z])$, together with symmetric structures. By hypothesis, $C$ and $D$ are torsion with respect to $\Z[t,t^{-1}]$ and $(1-t)$ acts as an automorphism. Together, these two conditions are equivalent to being torsion with respect to the set $S$ of Alexander polynomials. So the hypotheses of \fullref{prop:extsequence} are satisfied and the result follows.
\end{proof}

\begin{definition} The pairing obtained in the proof of \fullref{cor:blanchfield} is called the \emph{Blanchfield pairing} of $(Y,\partial Y)$, or simply the \emph{Blanchfield pairing of $Y$} if $\partial Y=\emptyset$.
\end{definition}

\begin{theorem}\label{thm:A-B-submodules} Given a strong algebraic slice $(f\colon C_K\to D,(\delta\phi_K,\phi_K))$ for an $n$-knot $K$, the images $A^i:=\im(f^*\colon H^i(D;\Z[\Z])\to H^i(C_K;\Z[\Z]))$ are such that $Bl(FA^i\times FA^{n+3-i})=0$.
\end{theorem}

In the proof of Theorem~\ref{thm:A-B-submodules}, we will use the following lemma, which is a straightforward functoriality consequence of the construction of the sequence in \fullref{prop:extsequence}. We omit the proof, which is standard homological algebra.

\begin{lemma}\label{lem:UCT}Let $R$, $S$, $G$ be as in \fullref{prop:extsequence}. Let $f\colon X\to Y$ be a morphism of projective left $R$-module chain complexes with $\Hom_R(H_p(X;R),G)=0=\Hom_R(H_p(Y;R),G)$ for all $p\in\Z$. The following is a commutative diagram for all $r\in\Z$
\[
\begin{tikzcd}
0\to \Ext^2_R(H_{r-2}(Y;R),G)\ar{r} \ar{d}{\Ext^2_R(f_*,G)} & H^{r}(Y;R) \ar{r} \ar{d}{f^*} & \Ext^1_R(H_{r-1}(Y;R),G) \ar{d}{\Ext^1_R(f_*,G)}\to 0\\
0\to \Ext^2_R(H_{r-2}(X;R),G)\ar{r} & H^{r}(X;R) \ar{r} & \Ext^1_R(H_{r-1}(X;R),G)\to 0
\end{tikzcd}
\]
where the rows are the short exact sequences of  \fullref{prop:extsequence}.
\end{lemma}

\begin{proof}[Proof of \fullref{thm:A-B-submodules}]

The pair $(f\colon C_K\to D,(\delta\phi_K,\phi_K))$ satisfies the hypotheses of \fullref{prop:blanchfieldwithcxs} (with $C'=0$). For any finitely generated ${\Z[\Z]}$-module $T$, define $T^\wedge:=\Hom_{\Z[\Z]}(T,{\Q(\Z)}/{\Z[\Z]})$. The following diagram has exact rows coming from the long exact sequence of the map $f\colon C_K\to D$ and vertical maps given by the adjoints to the various Blanchfield pairings, which are all isomorphisms:
\[
\begin{tikzcd}
FH^{i}(D;{\Z[\Z]})\arrow[r]\arrow[d, "Bl", "\cong"']
&FH^{i}(C;{\Z[\Z]})\arrow[r]\arrow[d, "Bl", "\cong"']
&FH^{i+1}(D,C;{\Z[\Z]})\arrow[d, "Bl", "\cong"']\\
(FH^{n+4-i}(D,C;{\Z[\Z]}))^\wedge\ar{r}
&(FH^{n+3-i}(C;{\Z[\Z]}))^\wedge\ar{r}
&(FH^{n+3-i}(D;{\Z[\Z]}))^\wedge
\end{tikzcd}
\]If the diagram is commutative, it is a standard diagram chase to show that $Bl(FA^i\times FA^{n+3-i})=0$.

To see that the diagram of Blanchfield maps above commutes, recall that the map $Bl$ from ~\fullref{prop:blanchfieldwithcxs} was constructed as a composite of four isomorphisms. The diagram of Blanchfield maps above can thus be decomposed using these four isomorphisms into a diagram with five rows (and three columns). Each row is a different exact sequence coming from the map $f\colon C_K\to D$. The diagram of Blanchfield maps commutes if the decomposed diagram commutes. But the only part of this decomposed diagram that is not well known to commute is the part which comes from the $UCT$ map. This part can be seen to commute by applying the functor $F$ to the diagram in  \fullref{lem:UCT}.
\end{proof}

By \fullref{cor:blanchfield}, for $k>1$, a $(2k-3)$-dimensional knot $K$ determines a middle-dimensional Blanchfield pairing
\[
Bl\colon FH^{k}(M_K;\Z[\Z])\times FH^{k}(M_K;\Z[\Z])\to \Q(\Z)/\Z[\Z].
\]
A submodule $j\colon L\hookrightarrow FH^{k}(M_K;\Z[\Z])$, is called a \emph{lagrangian} if the sequence
\[
0\to L\xrightarrow{j}FH^{k}(M_K;\Z[\Z])\xrightarrow{j^\wedge\circ (Bl)^{ad}} L^\wedge\to 0
\]
is exact, where as above `$\wedge$' denotes the functor $-^\wedge:= \Hom_{\Z[\Z]}(-,\Q(\Z)/\Z[\Z])$.

\begin{definition}\label{def:algslice}
An odd-dimensional knot $K$ is \emph{algebraically slice} if the middle-dimensional Blanchfield pairing admits a lagrangian. The knot $K$ is \emph{algebraically doubly slice} if the middle-dimensional Blanchfield pairing admits two lagrangians that are complementary as submodules.
\end{definition}

\noindent The following corollary is immediate from \fullref{prop:obvs} and \fullref{thm:A-B-submodules}.

\begin{corollary} If an odd-dimensional knot $K$ is strongly algebraically $($doubly$)$ slice then $K$ is algebraically $($doubly$)$ slice.
\end{corollary}

\begin{remark}
For $n>1$, an $n$-knot is strongly algebraically slice if and only if it is slice. One proof of this fact is that for $n>1$, there is an isomorphism from the $n$-dimensional knot concordance group to the symmetric $L$-group $\mathcal{C}_n\to L^{n+3}(\Z[\Z],S)$ given by sending $K$ to the class of $(C_K,\phi_K)$. An algebraic nullcobordism is then exactly the condition required for vanishing in $L^{n+3}(\Z[\Z],S)$.

Levine \cite{Levine-doubly-slice} constructed examples of knots whose cohomology rings failed to decompose into complementary subalgebras, with $\Q[\Z]$-coefficients, implying these knots are not strongly algebraically doubly slice. However as Levine observes in that paper, and is still the case, there are no known examples of this phenomenon for algebraically doubly slice knots. There are also no known examples of this phenomenon that use knots with $\pi_K\cong\Z$. It seems likely that understanding the difference between algebraically slice and strongly algebraically slice, for knots with $\pi_K\cong\Z$, could lead to a substantial characterisation of double-sliceness high-dimensionally. This would also clarify exactly how much abelian invariants can say about the doubly slice problem and where the techniques of Ruberman and of this paper become essential to the question.
\end{remark}

\bibliographystyle{alpha}
\bibliography{metabelian}

\newcommand{\etalchar}[1]{$^{#1}$}
\begin{thebibliography}{RCS{\etalchar{+}}96}

\bibitem[Bro67]{MR212816}
William Browder.
\newblock Diffeomorphisms of {$1$}-connected manifolds.
\newblock {\em Trans. Amer. Math. Soc.}, 128:155--163, 1967.

\bibitem[CG86]{Casson-Gordon:1986-1}
Andrew Casson and Cameron~McA. Gordon.
\newblock Cobordism of classical knots.
\newblock In {\em \`A la recherche de la topologie perdue}, pages 181--199.
  Birkh\"auser Boston, Boston, MA, 1986.
\newblock With an appendix by P. M. Gilmer.

\bibitem[Cha14]{MR3257550}
Jae~Choon Cha.
\newblock Amenable {$L^2$}-theoretic methods and knot concordance.
\newblock {\em Int. Math. Res. Not. IMRN}, (17):4768--4803, 2014.

\bibitem[Cha16]{MR3493628}
Jae~Choon Cha.
\newblock A topological approach to {C}heeger-{G}romov universal bounds for von
  {N}eumann {$\rho$}-invariants.
\newblock {\em Comm. Pure Appl. Math.}, 69(6):1154--1209, 2016.

\bibitem[CHL09]{Cochran-Harvey-Leidy:2009-1}
Tim~D. Cochran, Shelly Harvey, and Constance Leidy.
\newblock Knot concordance and higher-order {B}lanchfield duality.
\newblock {\em Geom. Topol.}, 13(3):1419--1482, 2009.

\bibitem[CHL11]{Cochran-Harvey-Leidy:2011-02}
Tim~D. Cochran, Shelly Harvey, and Constance Leidy.
\newblock Primary decomposition and the fractal nature of knot concordance.
\newblock {\em Math. Ann.}, 351(2):443--508, 2011.

\bibitem[CK08]{Cochran-Kim:2004-1}
Tim~D. Cochran and Taehee Kim.
\newblock Higher-order {A}lexander invariants and filtrations of the knot
  concordance group.
\newblock {\em Trans. Amer. Math. Soc.}, 360(3):1407--1441 (electronic), 2008.

\bibitem[CO12]{Cha-Orr:2009-01}
Jae~Choon Cha and Kent~E. Orr.
\newblock ${L}^2$-signatures, homology localization, and amenable groups.
\newblock {\em Comm. Pure Appl. Math.}, 65:790--832, 2012.

\bibitem[COT03]{MR1973052}
Tim~D. Cochran, Kent~E. Orr, and Peter Teichner.
\newblock Knot concordance, {W}hitney towers and {$L^2$}-signatures.
\newblock {\em Ann. of Math. (2)}, 157(2):433--519, 2003.

\bibitem[CS74]{Cappell-Shaneson:1974-1}
Sylvain~E. Cappell and Julius~L. Shaneson.
\newblock The codimension two placement problem and homology equivalent
  manifolds.
\newblock {\em Ann. of Math. (2)}, 99:277--348, 1974.

\bibitem[CW03]{MR1988288}
Stanley Chang and Shmuel Weinberger.
\newblock On invariants of {H}irzebruch and {C}heeger-{G}romov.
\newblock {\em Geom. Topol.}, 7:311--319, 2003.

\bibitem[Fox56]{Fox:1956-1}
Ralph~H. Fox.
\newblock Free differential calculus. {I}{I}{I}. {S}ubgroups.
\newblock {\em Ann. of Math. (2)}, 64:407--419, 1956.

\bibitem[FQ90]{Freedman-Quinn:1990-1}
Michael Freedman and Frank Quinn.
\newblock {\em Topology of 4-manifolds}, volume~39 of {\em Princeton
  Mathematical Series}.
\newblock Princeton University Press, Princeton, NJ, 1990.

\bibitem[Fra13]{Franklin:2013}
Bridget~D. Franklin.
\newblock The effect of infecting curves on knot concordance.
\newblock {\em Int. Math. Res. Not. IMRN}, (1):184--217, 2013.

\bibitem[Glu62]{Gluck}
Herman Gluck.
\newblock The embedding of two-spheres in the four-sphere.
\newblock {\em Trans. Amer. Math. Soc.}, 104:308--333, 1962.

\bibitem[Gor72]{Gordon-branched}
Cameron~McA. Gordon.
\newblock Knots whose branched cyclic coverings have periodic homology.
\newblock {\em Trans. Amer. Math. Soc.}, 168:357--370, 1972.

\bibitem[Gor78]{MR521730}
Cameron~McA. Gordon.
\newblock Some aspects of classical knot theory.
\newblock In {\em Knot theory ({P}roc. {S}em., {P}lans-sur-{B}ex, 1977)},
  volume 685 of {\em Lecture Notes in Math.}, pages 1--60. Springer, Berlin,
  1978.

\bibitem[Hau81]{MR627098}
Jean-Claude Hausmann.
\newblock On the homotopy of nonnilpotent spaces.
\newblock {\em Math. Z.}, 178(1):115--123, 1981.

\bibitem[Kat69]{Kato-concordance}
Mitsuyoshi Kato.
\newblock A concordance classification of {${\rm PL}$} homeomorphisms of
  {$S^{p}\times S^{q}$}.
\newblock {\em Topology}, 8:371--383, 1969.

\bibitem[Ker65]{Kervaire:1965-1}
Michel~A. Kervaire.
\newblock Les n\oe uds de dimensions sup\'erieures.
\newblock {\em Bull. Soc. Math. France}, 93:225--271, 1965.

\bibitem[Kim05]{Kim:2005-1}
Taehee Kim.
\newblock An infinite family of non-concordant knots having the same {S}eifert
  form.
\newblock {\em Comment. Math. Helv.}, 80(1):147--155, 2005.

\bibitem[Kim06]{Kim:2006-1}
Taehee Kim.
\newblock New obstructions to doubly slicing knots.
\newblock {\em Topology}, 45(3):543--566, 2006.

\bibitem[KS77]{MR0645390}
Robion~C. Kirby and Laurence~C. Siebenmann.
\newblock {\em Foundational essays on topological manifolds, smoothings, and
  triangulations}.
\newblock Princeton University Press, Princeton, N.J.; University of Tokyo
  Press, Tokyo, 1977.
\newblock With notes by John Milnor and Michael Atiyah, Annals of Mathematics
  Studies, No. 88.

\bibitem[Lev69]{MR0246314}
Jerome~P. Levine.
\newblock Knot cobordism groups in codimension two.
\newblock {\em Comment. Math. Helv.}, 44:229--244, 1969.

\bibitem[Lev77]{Levine-77-knot-modules}
Jerome~P. Levine.
\newblock Knot modules.
\newblock {\em Trans. Amer. Math. Soc.}, 229:1--50, 1977.

\bibitem[Lev83]{Levine-doubly-slice}
Jerome~P. Levine.
\newblock Doubly sliced knots and doubled disk knots.
\newblock {\em Michigan Math. J.}, 30(2):249--256, 1983.

\bibitem[Liv02]{Livingston:2002-1}
Charles Livingston.
\newblock Seifert forms and concordance.
\newblock {\em Geom. Topol.}, 6:403--408 (electronic), 2002.

\bibitem[L{\"u}c02]{Lueck:2002-1}
Wolfgang L{\"u}ck.
\newblock {\em {$L\sp 2$}-invariants: theory and applications to geometry and
  {$K$}-theory}, volume~44 of {\em Ergebnisse der Mathematik und ihrer
  Grenzgebiete. 3. Folge. A Series of Modern Surveys in Mathematics [Results in
  Mathematics and Related Areas. 3rd Series. A Series of Modern Surveys in
  Mathematics]}.
\newblock Springer-Verlag, Berlin, 2002.

\bibitem[Ors17]{MR3604378}
Patrick Orson.
\newblock Double {$L$}-groups and doubly slice knots.
\newblock {\em Algebr. Geom. Topol.}, 17(1):273--329, 2017.

\bibitem[Pat88]{Paterson:1988-1}
Alan L.~T. Paterson.
\newblock {\em Amenability}, volume~29 of {\em Mathematical Surveys and
  Monographs}.
\newblock American Mathematical Society, Providence, RI, 1988.

\bibitem[Ran80a]{MR560997}
A.~A. Ranicki.
\newblock The algebraic theory of surgery. {I}. {F}oundations.
\newblock {\em Proc. London Math. Soc. (3)}, 40(1):87--192, 1980.

\bibitem[Ran80b]{MR566491}
A.~A. Ranicki.
\newblock The algebraic theory of surgery. {II}. {A}pplications to topology.
\newblock {\em Proc. London Math. Soc. (3)}, 40(2):193--283, 1980.

\bibitem[Ran81]{MR620795}
A.~A. Ranicki.
\newblock {\em Exact sequences in the algebraic theory of surgery}, volume~26
  of {\em Mathematical Notes}.
\newblock Princeton University Press, Princeton, N.J., 1981.

\bibitem[RCS{\etalchar{+}}96]{hauptvermutung-book}
A.~A. Ranicki, A.~J. Casson, D.~P. Sullivan, M.~A. Armstrong, C.~P. Rourke, and
  G.~E. Cooke.
\newblock {\em The {H}auptvermutung book}, volume~1 of {\em $K$-Monographs in
  Mathematics}.
\newblock Kluwer Academic Publishers, Dordrecht, 1996.
\newblock A collection of papers of the topology of manifolds.

\bibitem[Rub83]{Ruberman-doubly-slice-I}
Daniel Ruberman.
\newblock Doubly slice knots and the {C}asson-{G}ordon invariants.
\newblock {\em Trans. Amer. Math. Soc.}, 279(2):569--588, 1983.

\bibitem[Rub88]{Ruberman-doubly-slice-II}
Daniel Ruberman.
\newblock The {C}asson-{G}ordon invariants in high-dimensional knot theory.
\newblock {\em Trans. Amer. Math. Soc.}, 306(2):579--595, 1988.

\bibitem[Str74]{Strebel:1974-1}
Ralph Strebel.
\newblock Homological methods applied to the derived series of groups.
\newblock {\em Comment. Math. Helv.}, 49:302--332, 1974.

\bibitem[Sum71]{MR0290351}
D.~W. Sumners.
\newblock Invertible knot cobordisms.
\newblock {\em Comment. Math. Helv.}, 46:240--256, 1971.

\end{thebibliography}
\end{document}